\documentclass[11pt]{article}

\usepackage{authblk}
\usepackage[margin=1in]{geometry}
\usepackage{xcolor}
\usepackage{hyperref}

\usepackage{amsfonts,amssymb}
\usepackage{amsrefs}
\usepackage{amsmath, verbatim,amsxtra}
\usepackage{amstext}
\usepackage{amsthm}
\usepackage{enumerate}

\newtheorem{thm}{Theorem}

\newtheorem*{clm*}{Claim}

\newtheorem{lem}{Lemma}[section]
\newtheorem{rmk}{Remark}[section]
\newtheorem{prop}[lem]{Proposition}
\newtheorem{obs}[lem]{Observation}
\newtheorem*{obs*}{Observation}

\newtheorem*{defin*}{Definition}
\newtheorem{cor}[lem]{Corollary}
\newtheorem*{cor*}{Corollary}

\theoremstyle{remark}

\newcommand{\ds}{\displaystyle}

\newcommand{\C}{\mathbb C}

\newcommand{\sinc}{\mathrm{sinc}}

\newcommand{\Ima}{\mathrm{Im}\,}
\newcommand{\Rea}{\mathrm{Re}\,}

\newcommand{\var}{\mathrm{var}\,}
\newcommand{\cov}{\mathrm{cov}\,}

\newcommand{\ind}{1{\hskip -2.5 pt}\hbox{I}}

\newcommand{\wh}{\widehat}
\newcommand{\ov}{\overline}
\newcommand{\norm}[1]{\left\Vert#1\right\Vert}
\newcommand{\conv}[1]{\overset{#1}{\longrightarrow}}

\newcommand{\B}{\bf B}
\newcommand{\OU}{\bf OU}
\newcommand{\W}{\bf W}

\newcommand{\R}{\mathbb R}

\newcommand{\E}{\mathbb E}
\newcommand{\N}{\mathbb N}
\newcommand{\Z}{\mathbb Z}
\newcommand{\Pro}{\mathbb P}
\newcommand{\lm}{\lambda}

\newcommand{\ep}{\varepsilon}

\newcommand{\al}{\alpha}

\newcommand{\D}{\Delta}
\newcommand{\calN}{\mathcal{N}}

\newcommand{\V}{\var[\D(T)]}

\title{The winding of stationary Gaussian processes}

\author{Jeremiah Buckley
\thanks{King's College, London, email: jeremiah.buckley@kcl.ac.uk. Research supported by funding from the European Research Council under the European Union's Seventh Framework Programme (FP7/2007-2013), ERC grant agreement n$^{\text{o}}$ 335141}
\hspace{2pt}  and Naomi Feldheim
\thanks{Stanford University, email: naomifel@stanford.edu. Research supported in part by the United States - Israel Binational Science Foundation, grant 2012037, and by NSF funds while the author was holding postdoctoral fellowships at the IMA Minnesota and at Stanford University.}
}

\begin{document}

\renewcommand{\thefootnote}{\fnsymbol{footnote}} 
\renewcommand{\thefootnote}{\arabic{footnote}}

\maketitle

\begin{abstract}
This paper studies the winding of a continuously differentiable Gaussian stationary process $f:\R\to\C$ in the interval $[0,T]$.
We give formulae for the mean and the variance of this random variable.
The variance is shown to always grow at least linearly with $T$, and conditions for it
to be asymptotically linear or quadratic are given.
Moreover, we show that if the covariance function together with its
 second derivative are in $L^2(\R)$, then the winding obeys a central limit theorem.
These results correspond to similar results for zeroes of real-valued stationary Gaussian functions by Cuzick, Slud and others.
\end{abstract}


\section{Introduction}

Gaussian functions on various spaces, and in particular stationary functions (i.e., those functions whose distribution is invariant under shifts), have long been an object of extensive study. Real Gaussian stationary functions $f:\R\to\R$ are a classical model of random signals, and in particular much effort was devoted to the study of their zeroes~\cites{AT, CL}. More recently, zeroes of complex Gaussian functions $f:\C\to\C$ attracted attention, as they are interesting point processes with intrinsic repulsion~\cite{GAFbook}.

In this paper we study the winding, or the increment of the argument, of planar Gaussian stationary processes $f:\R\to\C$. In light of the argument principle, one might expect winding to be the appropriate analogue in this setting of zeroes in the aforementioned examples; indeed our results and methods are closely related to the corresponding ones for random zeroes, both in the real~\cite{Slud} and complex~\cite{Feld} settings. In this sense, this work is part of an effort to simplify, unify and generalize the tools
which are used for analysing random zeroes.

In addition, this work is also motivated by a long history of works concerning the winding of various planar processes. Winding is used to model the entanglement of polymers~\cite{GF}, and the movement of a particle under a random magnetic field~\cite{DK}. Limit laws and asymptotic behavior of the winding were studied for Brownian motion~\cites{Spitzer, PY1} and certain fractal curves~\cite{WW}, among others. However, perhaps surprisingly, the winding of Gaussian stationary processes appears to be a topic that has been largely ignored. Prior to this work, we know only of a paper by Le-Doussal, Etzioni and Horovitz \cite{DEH} which provides predictions and intriguing examples regarding the nature of the fluctuations of the winding. Their interest was inspired by their research on the winding of particles in random environments~\cite{EHD2012}. The present paper establishes and extends their predictions. More about background and motivation may be found in Section~\ref{sec: discuss}.

We now give a brief overview of our results. Denote by $\Delta (T)$ the winding (or increment of the argument) of $f$ around the origin in the ``time''-interval $[0,T]$. We develop an asymptotic formula for the variance $V(T)=\V$ of the winding of $f$ in ``time'' $[0,T]$ (Theorem~\ref{thm: var}). By analysing this formula, we show that $V(T)$ is always at least linear in $T$ (Theorem~\ref{thm: asym}). Then we prove that if the covariance function and its second derivative are in $L^2$, then $V(T)$ is asymptotically linear in $T$ and a central limit theorem holds (Theorem~\ref{thm: CLT}). Finally we show that if the spectral measure of the Gaussian process $f$ does not contain any atoms, then $V(T)$ is sub-quadratic (Theorem~\ref{thm: sub-quad}).

\subsection{Definitions}
A \emph{standard complex Gaussian}, denoted $\calN_\C(0,1)$, is a $\C$-valued random variable whose distribution has density $ \frac 1 \pi e^{-|z|^2}$ against Lebesgue measure
on the plane. A \emph{complex Gaussian vector} is a random vector in $\C^n$ that is equal in distribution to $A\mathbf{v}$, where $\mathbf{v}$ is a random vector in $\C^m$ whose components are i.i.d. $\calN_\C(0,1)$-distributed, and $A$ is an $n \times m$ matrix (we always consider \emph{centred} random variables and processes, i.e., having mean $0$).

A \emph{complex Gaussian process} $f:\R\to\C$ is a random process whose finite marginals are Gaussian vectors; that is, for any $n\in \N$ and any $t_1,\dots, t_n\in \R$ the vector $(f(t_1), \dots, f(t_n))$ is a complex Gaussian vector.
Such a process is \emph{stationary} if its distribution is invariant under all real shifts, that is, for any $n\in\N$, $t_1,\dots, t_n\in \R$ and $s\in \R$ the vectors
$(f(t_1),\dots, f(t_n))$ and $(f(t_1+s),\dots, f(t_n+s) )$ have the same distribution. We will write GSP to denote a Gaussian stationary process throughout this article.

It is well-known that a GSP is determined by its \emph{covariance kernel} $r:\R\to \C$, given by
\[
r(t)=\E [f(t) \overline{f(0)} ].
\]
We normalize the process to have unit variance, that is, $r(0)=1$, which implies that $|r(t)|\leq1$ for all $t$ (see \eqref{eq: r<1}). We assume throughout that $r$ is continuous (in fact, we will assume \eqref{eq: log moment} below which is much stronger). Since $r$ is positive-definite and continuous, it follows from Bochner's theorem that it is the Fourier transform of some probability measure $\rho$ on the real line;
\begin{equation}\label{eq: rho}
r(t) = \wh{\rho}(t)=\int_\R e^{-i t \lambda}d\rho(\lambda).
\end{equation}
The measure $\rho$ is called the \emph{spectral measure} of $f$. Throughout this article we assume that for some $\al>0$
\begin{equation}\label{eq: log moment}
  \int_{\R} \lm^2 \log^{1+\al} (1+|\lm|) d\rho(\lm)<\infty.
\end{equation}
This condition ensures that $f$ is a.s. continuously differentiable, and is in fact close to being necessary (see \cite{AT}*{Ch. 1.4.1}). This condition also implies that $r$ is twice differentiable and that (see \cite{CL}*{Chapter 9, Lemma 1}) for all $\beta\leq\al$ and $C>0$ we have
\begin{equation}\label{eq: r'' at 0}
  |r''(t)-r''(0)|\leq\frac{C}{|\log|t||^{\beta}}\quad\text{for }|t|\leq\delta(\beta,C).
\end{equation}
This in turn implies that $r''$ is continuous on all of $\R$.

We define the \emph{winding} (or \emph{the increment of the argument}) of $f$ in the interval $[0,T]$ to be
\begin{equation}\label{eq: def D}
\D(T) = \frac 1 {2i} \int_{0}^T \left[ \frac {f^\prime(t)}{f(t)} - \overline{\left(\frac {f^\prime(t)}{f(t)} \right)}  \right] dt.
\end{equation}
This definition becomes clear when we recall that
almost surely $f$ does not vanish on the interval $[0,T]$, and so a branch of $\log f$ may be defined locally.
Then the infinitesimal change in the increment of the argument is given by
\[
d( \text{arg} f(t)) = d( \Ima \log f(t)) = \Ima  \frac {f'(t)}{f(t)}\,dt,
\]
which does not depend on the branch of $\log f$ we have chosen.

A GSP is \emph{degenerate} if its spectral measure consists of exactly one
atom $\delta_\phi$. In this case, the process may be represented as $f(t)=\zeta \cdot e^{-i\phi t}$ where $\zeta$ is a complex Gaussian random variable.
Thus $f$ moves along a circle with random radius and random starting phase, but
with constant (non-random) angular speed. The winding is thus deterministically $\Delta(T)= -\phi T$, and we
exclude this case from our studies.

\subsection{Results}
In all of our results we assume that $f:\R\to\C$ is a non-degenerate GSP
whose spectral measure obeys condition~\eqref{eq: log moment}.
The first result gives explicit formulae for the mean and variance of $\Delta(T)$.
\begin{thm}\label{thm: var}
For any $T>0$ we have:
\begin{enumerate}
\item
$\E[\D(T)] = T \ \Ima r'(0)$.
\item Denoting $R(x)=\frac {r'}{r}(x)$ for $x$ such that $r(x)\neq0$, we define $K:\R\to\R$ by
\begin{equation}\label{eq: K}
K(x)=\begin{cases}
\frac 1 2 |r'(x)|^2, & \text{if } |r(x)|=0\text{ or }1,\\
 \frac{|r(x)|^2}{1-|r(x)|^2}\ \Ima^2 \left\{R(x)-R(0) \right\} - \frac 1 2 \log \left( \frac 1 {1-|r(x)|^2}\right) \Rea\{R'{(x)}\},
&  \text{if } 0 < |r(x)| < 1.
\end{cases}
\end{equation}
Then $K$ is integrable on any compact subset of $\R$, and
\begin{equation}\label{eq: var integral in statement}
\var [\D(T) ] = T \int_{-T}^T \left( 1-\frac{|x|}{T} \right) K(x) dx.
\end{equation}
\end{enumerate}
\end{thm}

\begin{rmk}
{\rm
It is not hard to see that $K(x)$ is continuous at the points where $r(x)=0$ (which may be a large set). On the other hand, there is no natural definition of $K(x)$ at the points where $|r(x)|=1$, and we have assigned the value $\frac 1 2 |r'(x)|^2$ purely for convenience. In the course of the proof we will show that these points are isolated and that $K$ has a logarithmic, integrable singularity at each of them.
}
\end{rmk}

\begin{rmk}
{\rm
One may check that the kernel $K$ is always non-negative, but we will not reproduce the calculations here since they will not be important for our purposes. An alternative form for the variance, which may be more convenient for applications, will be given in the course of the paper (see Proposition~\ref{prop: IBP} below) - the kernel $\widetilde{K}$ given there is trivially non-negative.
}
\end{rmk}

\begin{rmk}\label{rmk: sym}
{\rm
The expression in Theorem~\ref{thm: var} simplifies under the assumption that $r$ is real-valued, which means that $r(t)=r(-t)$ for all $t\in\R$, or equivalently that $\rho$ is symmetric (i.e., $\rho(-I)=\rho(I)$ for all Borel subsets $I\subset\R$). This is sometimes called \emph{reflectional symmetry}. In this case (using the reformulation given by Proposition \ref{prop: IBP}) we get
\begin{align*}
\frac{\V}{T}
&= \int_{-T}^T \left( 1-\frac{|x|}{T} \right) \frac {r'(x)^2}{1-r(x)^2} dx + O\left(\frac1T\right), \: \text{ as } T\to\infty.
\end{align*}
In particular,
\[
\frac{\V}{T} =  \int_{-\infty}^\infty \frac {(r')^2}{1-r^2} + o(1), \: \text{ as } T\to\infty,
\]
provided that this infinite integral is finite. This coincides with predictions in a physics paper of Le Doussal, Etzioni and Horovitz~\cite{DEH}.
They also noticed the following simplification in this case: denoting $\theta (x) = \arcsin r(x)$ ($\theta$ is well-defined, since $r$ is now real-valued), we have
\[
\frac{\V}{T} = \int_{-\infty}^\infty (\theta')^2   + o(1), \: \text{ as } T\to\infty.
\]
}
\end{rmk}

\begin{rmk}
{\rm
Although the main focus of this paper is the ``large-time'' asymptotic behaviour, Le Doussal, Etzioni and Horovitz~\cite{DEH} also mention the short-time asymptotics of $\V$. Our result implies that $\V\sim(r'(0)^2-r''(0))T\log\tfrac1T$ as $T\to0$ and further terms in the asymptotic expansion may be obtained if one assumes some extra regularity - the existence of higher order derivatives of $r$. In Lemma~\ref{lem: |r|=1} we show that $r'(0)^2-r''(0)>0$.
}
\end{rmk}

Our next theorem states that the variance always grows at least linearly.
\begin{thm}\label{thm: asym}
There exists $C>0$ (depending on the covariance function $r$) such that
\[
\V \ge C \cdot T.
\]
\end{thm}

The case of asymptotically linear variance is of particular interest. Below we give a simple condition that is sufficient for this to hold, and prove a central limit theorem (CLT) under this hypothesis.
\begin{thm}\label{thm: CLT}
If $r, r' \in L^2(\R)$, then
\begin{equation}\label{eq: lin limit}
\lim_{T\to\infty} \frac{\var[\D(T)]}{T} \in (0,\infty).
\end{equation}
If in addition $r''\in L^2(\R)$, then $\D(T)$ obeys a central limit law, that is,
\begin{equation}\label{eq: clt}
\frac{\D(T)-\E[\D(T)]}{\sqrt{\var [\D(T)] }} \rightarrow \calN_\R(0,1),
\end{equation}
in distribution as $T\to\infty$.
\end{thm}

\begin{rmk}\label{rmk: r'}
{\rm
If $r,r''\in L^2(\R)$ then also $r'\in L^2(\R)$ (see Observation~\ref{obs: r'}). Therefore, the condition $r, r''\in L^2(\R)$ is enough to ensure both linear variance and a CLT.
}
\end{rmk}

On the other hand, the variance is trivially at most quadratic in $T$. The following theorem gives a mild mixing condition for the variance to be sub-quadratic.

\begin{thm}\label{thm: sub-quad}
If the spectral measure $\rho$ has no atoms, then
\[
\lim_{T\to\infty} \frac{\V}{T^2} = 0.
\]
\end{thm}

This was already proved in~\cite{Feld}, but we repeat the proof at the end of this paper for completeness. We note that, under the assumption that $f$ a.s. has an analytic extension to a strip in the complex plane, the converse to Theorem~\ref{thm: sub-quad} holds (see~\cite{Feld}*{Remark 1.5}).

The rest of the paper is organized as follows. Section~\ref{sec: discuss} is devoted to a discussion of motivation, related previous work and interesting examples. In Section~\ref{sec: mean + var} we prove Theorem~\ref{thm: var} about the mean and variance. In Section~\ref{sec: asym} we prove Theorem~\ref{thm: asym} (concerning a lower bound for the variance), after developing an alternative form for the variance (Proposition~\ref{prop: IBP}). In Section~\ref{sec: lin+clt} we prove Theorem~\ref{thm: CLT} concerning linear variance and a CLT. Finally, Section~\ref{sec: sub-quad} contains the proof of Theorem~\ref{thm: sub-quad} about sub-quadratic variance.

Finally, a word about notation. By $g\lesssim h$ we mean that $g\le C\cdot h$, where $C>0$ is a constant (which may vary from line to line, and may depend on fixed parameters). We write $g=O(h)$ if $|g|\lesssim h$.
Similarly, $g\simeq h$ means that $g\lesssim h$ and $h\lesssim g$. We use the notation $g(T) \asymp h(T)$ to denote that $\lim_{T\to\infty} \frac{g}{h}(T)$ exists and is some finite positive constant, while we write $g(T) \sim h(T)$ to denote the more precise $\lim_{T\to\infty} \frac{g}{h}(T)=1$

\section{Discussion}\label{sec: discuss}
\subsection{Background and motivation}
There are three major motivations for this work. The first comes from theoretical physics, where the winding of planar random processes is used in models of polymers, flux lines in superconductors and the quantum Hall effect (see~\cites{DK, GF, Vakpoly} and the references therein). For this reason, and out of pure mathematical interest, the winding has been studied for certain processes. For planar Brownian motion $\B$, Spitzer~\cite{Spitzer} proved, denoting the winding of $\B$ up to time $T$ by $\D_{\B}(T)$, that $\D_{\B}(T) / \log T$ converges in distribution to a Cauchy random variable. This inspired a long sequence of works (most notably, Pitman-Yor~\cites{PY1, PY2}). There was also much interest in windings of various fractal random curves (e.g. SARW~\cite{S}, SLE and related processes~\cite{WW}).
Very recently, winding of Ornstein-Uhlenbeck processes~\cite{VakOU} and of stable processes~\cite{DV} were studied, including analysis of large scale asymptotics and limit laws.
Some other relatively recent studies of winding with physical applications include \cites{DB, HD, KCM, GF}.

Le Doussal-Etzioni-Horovitz \cite{DEH} have studied the winding of planar Gaussian processes. The authors provide a formula for the variance of the winding of a Gaussian process, not necessarily stationary, with reflectional symmetry. Theorem~\ref{thm: var} of this paper is a rigorous derivation of the same formula for stationary processes, without assuming reflectional symmetry. We comment that it is possible to apply our methods to non-stationary processes as well, but we did not pursue this route. Le Doussal-Etzioni-Horovitz also noticed ``diffusive behavior'' (i.e., that the variance grows at least linearly) in all examples of interest, which led them to predict that ``for most stationary processes the winding angle exhibits diffusion''. Theorem~\ref{thm: asym} establishes this fact \emph{for all} sufficiently smooth processes.

The second motivation for this work is the extensive study of the zeroes of real stationary Gaussian processes $f:\R\to\R$. Morally, in many scenarios zeroes are analogous to winding (related, for instance, by the argument principle). The survey~\cite{Kratz} gives a good account of the research on zeroes of real GSPs, and we rely on it for details and references in what follows. The mean number of zeroes was computed by Kac, while asymptotics of the variance were studied by Cramer-Leadbetter, Piterbarg and many others; however, no accessible formula for the variance was given. For this reason, the first CLTs contained conditions about the variance which were hard to check. One such example is the work of Cuzick~\cite{Cuz}, who proved a CLT whose main condition is linear growth of the variance. Our proof of the CLT in Theorem~\ref{thm: CLT} is inspired by his, where, using our formula from Theorem~\ref{thm: var}, we can give an explicit condition for this linear growth. It is interesting to note that, after many years, Slud~\cite{Slud} gave a condition for linear growth of the variance of the number of zeroes, which is similar to the one we recovered for the winding in Theorem~\ref{thm: CLT}; i.e., that the covariance function and its second derivative are in $L^2(\R)$ (see remark~\ref{rmk: r'}). However, while in this article we analyse a concrete formula, Slud's work relies on sophisticated stochastic integration. We do not know of a way to unify these results. We note that Cuzick's and Slud's results were used by Granville-Wigman~\cite{GW} to study the variance and CLT for the zeroes of random trigonometric polynomials.

The third motivation comes from the study of complex zeroes of random Gaussian analytic functions. These have drawn increasing attention in recent years, as they provide rich and accessible point processes in the plane (see the recent book~\cite{GAFbook}). One of us \cite{Feld} proved very similar results to ours about fluctuations of complex zeroes of stationary Gaussian analytic functions (without a CLT). While, once again, the methods are different and a priori neither result implies the other, the variance is shown to always be at least linear (as in Theorem~\ref{thm: asym}), and the condition given for asymptotic linearity is very similar to ours (as in Theorem~\ref{thm: CLT}). The proof of sub-quadratic variance here (Theorem~\ref{thm: sub-quad}) is identical to that of \cite{Feld}.

We end by posing two natural open problems. The first is to determine the asymptotic behavior of the winding in case of non-linear variance (in particular, when the conditions of Theorem~\ref{thm: CLT} do not hold).
In similar cases for random real zeroes, it was shown that there are regimes of CLT and regimes of non-CLT behavior~\cite{Slud94}*{Thm 3.2}. However, the tools used for these results are, again, quite sophisticated and specific.
The second is to prove a converse to Theorem~\ref{thm: sub-quad} with no further assumptions (that is, that if the spectral measure contains an atom, then the variance is quadratic).

\subsection{Examples}
In this section we discuss some interesting GSPs. The last two examples were pointed out by Le Doussal, Etzioni and Horovitz~\cite{DEH}. We stress that, while here we present only orders of magnitude for $\V$ in the various examples, often one may apply our results to retrieve exact constants.

{\bf Atomic spectral measure.} As mentioned earlier, if $\rho$ consists of one atom $\delta_\phi$ then $f(t)=\zeta\, e^{-i\phi t}$ is a degenerate circular motion. However, a superposition of such processes results in a \emph{random almost periodic function}, with highly non-trivial behavior. Specifically, consider $\rho=\sum_j c_j^2 \delta_{\phi_j}$ where $c_j,\phi_j\in \R$ are given, $\sum_j c_j^2=1$ and $\sum_j c_j^2 \phi_j^2<\infty$. The corresponding process is
\begin{equation}\label{eq: atomic}
f(t) = \sum_j c_j \zeta_j\, e^{-i \phi_j t}, \quad \{\zeta_j\} \text{ are i.i.d. } \calN_\C(0,1).
\end{equation}
This is a stochastic version of the famous Lagrange model in celestial mechanics, as $f$ represents the motion of the ``tip" of a chain of celestial bodies, each circling the previous one with fixed angular speed and radius. The classical ``Lagrange problem'' was to show that, for any (deterministic) choice of $c_j, \phi_j$ and $\zeta_j$, the winding of $f$ should exhibit a \emph{mean motion}, i.e., that the limit $\omega_f:=\lim_{T\to\infty}\frac{\Delta(T)}{T}$ should always exist. After many years of research it was proved by Jessen and Tornehave~\cite{JT}*{Theorem 26}, but the concrete dependence of the limit on the parameters is not fully understood.

For a random $f$ as in~\eqref{eq: atomic}, we obtain $\var[\Delta(T)] \asymp T^2$. To see this, first notice that $f$ is a.s. bounded (since it is almost periodic). Thus, for any $j$ the event
\begin{equation*}
  A_j=\left\{|c_j \zeta_j| \ge 2\sup_{t\in\R}\Big|\sum_{k\ne j} c_k \zeta_k e^{-i\phi_k t}\Big| \right\}
\end{equation*}
has positive probability, and therefore so too does the event $\{\Delta(T) = -\phi_j T +O(1)\text{ as }T\to\infty\}$. As a consequence, the limiting distribution of $\frac{\D(T)}{T}$ contains atoms at the points $-\phi_j$, whose mass is at least $\Pro(A_j)$ (independent of $T$). This yields that $\var\left(\frac{\D(T)}{T}\right) \ge C$, and moreover, that a CLT does not hold.
Similar properties were observed for complex zeroes of the sum~\eqref{eq: atomic} (allowing $t\in \C$), see~\cite{F0}*{Remark 2.3}.

{\bf Sinc kernel.} Taking $\rho = \tfrac1{2\pi}\ind_{[-\pi,\pi]}$ one obtains $r(t)=\sinc(t)=\frac{\sin(\pi t)}{\pi t}$. This process has the representation $f(t) = \sum_{n\in\Z} \zeta_n \sinc(t-n)$, where $\{\zeta_n\}_{n\in\Z}$ are i.i.d. $\calN_\C(0,1)$. Notice that $f(n)=\zeta_n$ for $n\in\Z$, so this process may be regarded as a smooth (in fact, analytic) interpolation of the i.i.d. sequence.
For this example, Theorem~\ref{thm: CLT} yields that $\V\asymp T$, and a CLT holds.

{\bf Gaussian kernel.}
Take $r(t)=e^{-t^2/2}$, or equivalently a Gaussian spectral measure. The corresponding process can be written in the form
\[
f(t)=e^{-t^2/2}\sum_{j=0}^\infty \zeta_j \frac{t^j}{\sqrt{j!} }, \quad \{\zeta_j\} \text{ are i.i.d. } \calN_\C(0,1).
\]
Theorem~\ref{thm: CLT} may be applied to obtain $\V \asymp T$ and a CLT for the winding.

{\bf Exponential kernel and approximations.}
Consider $r_{\OU}(t)=e^{-|t|}$. This process is a time-space change of Brownian motion, called \emph{the Ornstein-Uhlenbeck (OU) process}. Inspired by Spitzer's limit law for $\Delta_{\B}$, Vakeroudis \cite{VakOU}*{Theorem 3.3} has recently shown that $\frac{\Delta_{\OU}(T)}{T}$ converges in distribution to the Cauchy law; in particular the variance of the winding in each finite interval is infinite. As the OU process is not differentiable, none of our results may be directly applied. However, one may approximate the OU process by differentiable processes. One way to do so is by taking $r_a(t) = e^{a-\sqrt{a^2+t^2}}$ with $a\downarrow 0$. For a fixed $a>0$, since $r_a$ is infinitely differentiable, we may apply Theorem~\ref{thm: var} to see that the variance of the winding of the corresponding process in $[0,T]$ is of order $\ln(\frac 1 a)\cdot T$ for $T\geq a^{1-\ep}$. As $a\to 0$ we see that the variance is unbounded, and this holds even on certain short intervals that are not ``too short''.

Another approximation may be derived using the spectral measure. The OU process has spectral density $\frac 1{\pi(1+\lm^2)}$, thus one may consider the spectral density $\frac{M}{\pi(M-1)}\left(\frac 1 {\lm^2+1}-\frac 1{\lm^2+M^2} \right)$ which approximates the OU process as $M\to\infty$, and satisfies \eqref{eq: log moment} for each fixed $M$. The corresponding covariance kernel is $r_M(t)=\frac{ M e^{-|t|} - e^{-M|t|}}{M-1}$, which is twice differentiable. Applying Theorem~\ref{thm: var} one gets a variance of size $\ln M \cdot T$ for $T\geq M^{-1+\ep}$, and again we see that as $M\to\infty$, the variance is unbounded, even on certain short intervals.

{\bf Other intermediate growth examples.}
Using the formulae in Theorem~\ref{thm: var} or Remark~\ref{rmk: sym}, it is possible to construct examples where $T \ll \V \ll T^2$. For instance:
\begin{itemize}
\item For $r(t)=J_0(t)$, one has $\V \asymp T\ln T$. Here $J_0$ stands for the $0$-Bessel function of the first kind.
\item Let $0<b<\frac 1 2$. For $r(t)=\frac{\cos t}{(1+|t|)^b}$, one has $\V\asymp T^{2-2b}$.
\end{itemize}


\section{Formulae for the Mean and Variance: Theorem~\ref{thm: var}}\label{sec: mean + var}

\subsection{Preliminaries}
In the course of the proof of Theorem~\ref{thm: var} we shall make use of the following lemmata.
The first is an extension of an exercise in Kahane's celebrated book \cite{Kah}*{Ch. XXII, Ex. 3}.

\begin{lem}\label{lem: K}
Let $F_1, F_2, F'_1, F'_2$ be jointly Gaussian complex random variables. Let
\[
r_{jk} = \E[F_j\overline{F_k}], \quad s_{jk} = \E[F'_j \overline{F_k}], \quad t_{jk} = \E[F'_j \overline{F'_k}].
\]
Assume that $r_{11}, r_{22} \neq 0$ and that $r_{11}r_{22} \neq  |r_{12}|^2$. Then:
\begin{enumerate}[(a)]
\item $ \E \left[\frac{F'_1}{F_1} \right] = \frac{s_{11}}{r_{11}}.$
\item If $r_{12}\neq 0$, then
\[
\cov \left( \frac{F'_1}{F_1}, \frac{F'_2}{F_2} \right) =
\frac{|r_{12}|^2}{r_{11} r_{22}-|r_{12}|^2} \left(\frac{s_{12}}{ r_{12}} - \frac{s_{11}}{r_{11}}\right)\left(\frac{s_{21}}{r_{21}}-\frac{s_{22}}{r_{22}}\right)
\]
while if $r_{12}=0$, then $\cov\left( \frac{F'_1}{F_1}, \frac{F'_2}{F_2} \right)=\frac{s_{12} s_{21}}{r_{11} r_{22}}$.
\item If $r_{12}\neq 0$, then
\begin{align*}
\cov\left( \frac{F'_1}{F_1}, \overline{\left(\frac{F'_2}{F_2}\right)} \right)=
\frac{|r_{12}|^2}{r_{11} r_{22}-|r_{12}|^2}& \left(\frac{s_{12}}{ r_{12}} - \frac{s_{11}}{r_{11}}\right)\overline{\left(\frac{s_{21}}{r_{21}}-\frac{s_{22}}{r_{22}}\right)}\\
&+ \log \left(\frac{r_{11} r_{22}}{ r_{11}r_{22}-|r_{12}|^2 } \right)\cdot \left(\frac{t_{12}} {r_{12}} - \frac{s_{12}\overline{s_{21}} } {(r_{12})^2}\right),
\end{align*}
while if $r_{12}=0$ then $\cov\left( \frac{F'_1}{F_1}, \overline{\left(\frac{F'_2}{F_2}\right)} \right)=0$.
\end{enumerate}
\end{lem}

\begin{rmk}\label{rmk: K}
{\rm
If we fix all of the parameters except for $r_{12}$ (and $r_{21}=\overline{r_{12}}$), then the covariances computed in (b) and (c) are continuous functions of $r_{12}$ (i.e., at $r_{12}=0$).

If we drop any of the assumptions $r_{11}\neq 0$, $r_{22}\neq 0$ or $r_{11}r_{22}\neq |r_{12}|^2$ then the quantities computed in (b) and (c) diverge. We only require $r_{11}\neq 0$ for (a) to be finite.
}
\end{rmk}

All three parts of Lemma~\ref{lem: K} are proved in a similar way, which we outline below.

\begin{proof}[Sketch of the proof of Lemma~\ref{lem: K}]
Let $\zeta_1, \zeta_2$ be two i.i.d. $\calN_\C(0,1)$ random variables.
Since $F_1, F'_1$ are jointly Gaussian, there are $\alpha,\beta, \gamma \in \C$ such that the pair $(F_1,F'_1)$ has the same distribution as $(\alpha \zeta_1,\beta\zeta_1+\gamma \zeta_2)$.
In particular,
\begin{align*}
\frac{F'_1}{F_1} \overset{d}{=}  \frac{\beta}{\alpha} + \frac{\gamma}{\alpha} \frac {\zeta_2}{\zeta_1}.
\end{align*}
Taking expectation, and recalling that $\E\left[\frac {\zeta_2}{\zeta_1} \right] = \E [\zeta_2] \E\left[\frac 1{\zeta_1}\right] = 0$, we get
\[
\E \left[ \frac{F'_1}{F_1} \right]=  \frac{\beta}{\alpha}.
\]
All that remains is to express the parameters $\al$ and $\beta$ using $r_{jk}$, $s_{jk}$ and $t_{jk}$ for $j,k\in\{1,2\}$. This is done through covariance computations, namely
\[
|\alpha|^2 = \E[|F_1|^2]=r_{11}, \quad \text{ and } \quad \beta\ov{\alpha} =\E[F'_1 \ov{F_1} ]=s_{11}.
\]
Thus $\frac{\beta}{\alpha} =\frac {s_{11}}{r_{11}}$, which completes the proof of part (a).

As the proofs of the remaining cases are long but contain no new ideas, we omit them from this paper.
\end{proof}

Next we note some basic properties of the covariance function.
\begin{obs}\label{obs: sym}
Let $r:\R\to\C$ be a covariance function which is twice-differentiable. Then
\begin{equation}\label{eq: symmetries}
r(-x) = \overline{r(x)},  \quad r'(-x)=-\overline{r'(x)} \quad\text{and}\quad r''(-x)=\overline{r''(x)}.
\end{equation}
Also, for $R(x) = \frac{r'(x)}{r(x)}$ (at the points where $r(x)\neq0$), we have
\begin{equation}\label{eq: R symmetries}
R(x) = - \overline{R(-x)}, \quad R'(x) = \frac{r''(x)}{r(x)} - \frac{r'(x)^2}{r(x)^2} \quad\text{and}\quad R'(x) = \ov{ R'(-x) }.
\end{equation}
\end{obs}

\begin{proof}
Recalling that $r$ is the Fourier transform of a probability measure (as in~\eqref{eq: rho}), we get that $r(-x)=\overline{r(x)}$.
All other relations follow easily from this.
\end{proof}

The next lemma will allow us to analyse the behavior of $r$ near its extremal points.

\begin{lem}\label{lem: |r|=1}
Let $r:\R\to\C$ be the covariance function of a non-degenerate GSP, normalized so that $r(0)=1$.
The following hold:
\begin{itemize}
\item For all $t\in\R$, $|r(t)|\le 1=r(0)$.
\item If there exists $t\neq 0$ such that $|r(t)|=1$, then there exists $\lm_0, \lm_1\in \R$ such that $\text{sprt}(\rho)\subseteq \lm_0 + \lm_1 \Z$.
\item The set $D=\{t: \; |r(t)| = 1\}$ is discrete.
\item If $r$ is twice differentiable, then there exists $C>0$ such that for any $t_m\in D$
\[
1-|r(t)|^2 = C(t-t_m)^2 + o\left( (t-t_m)^2\right), \text{ as } t\to t_m.
\]
\end{itemize}
\end{lem}

\begin{proof}[Proof of Lemma~\ref{lem: |r|=1}]
For any fixed $t\in \R$, from \eqref{eq: rho} we have
\begin{equation}\label{eq: r<1}
|r(t)| = \left|\int_\R e^{-i\lm t} d\rho(\lm)\right| \le \int_\R 1 \cdot d\rho(\lm) = r(0) = 1,
\end{equation}
which proves the first item.
Equality in \eqref{eq: r<1} holds if and only if $e^{-i \lm t}$ is a constant in $L^2(\rho)$. If $t\neq 0$ this holds if and only if $\text{supp}(\rho)\subseteq \lm_0 + \frac{2\pi}{t} \Z$. This establishes items 2 and 3 of our lemma.

Suppose now that $D\neq \{0\}$, and write $\rho = \sum_{n\in\Z} a_n \delta_{\lambda_0+n\lambda_1}$ where $a_n\geq 0$, $\sum_n a_n=1$ and there are at least two integers $n$ such that $a_n\neq 0$. Then
\[
r(t) = e^{- i \lm_0 t} \sum_{n\in\Z} a_n e^{- i n \lambda_1 t},
\]
and one may verify that for any $t_m\in D$
\[
r^{(k)}(t_m) = e^{-i\lambda_0 t_m}r^{(k)}(0),
\]
for any $k\in\N$ for which $r^{(k)}$ is defined.
The second order Taylor expansion for $|r|^2$ gives
\begin{align*}
|r(t)|^2
&= |r(t_m)|^2 + \left( r(t_m)\ov{r'(t_m)} + \ov{r(t_m)} r'(t_m) \right) (t-t_m) \\
& + \left( \frac 1 2 r(t_m) \ov{r''(t_m)} + |r'(t_m)|^2 +\frac 1 2 \ov{r(t_m)} r''(t_m) \right)(t-t_m)^2 + o\left( (t-t_m)^2 \right)\\
&=1 + \left( \ov{r'(0)} +  r'(0) \right) (t-t_m) \\
& + \left( \frac 1 2  \ov{r''(0)} + |r'(0)|^2 +\frac 1 2  r''(0) \right)(t-t_m)^2 + o\left( (t-t_m)^2 \right)\\
&=1 + \left(r''(0)-r'(0)^2\right) (t-t_m)^2 +  o\left( (t-t_m)^2 \right),
\end{align*}
where the last simplifications are due to \eqref{eq: symmetries}. The same conclusion holds for $t_m=0$ in the case $D=\{0\}$. It remains to prove that
$C=r'(0)^2-r''(0)\ge 0$. Indeed, using \eqref{eq: rho} once more we have:
\begin{align*}
-r'(0)^2  &= -\left( \int_\R i\lm d\rho(\lm) \right)^2  = \left( \int_\R \lm d\rho(\lm) \right)^2  \\
&\le \left( \int_\R \lm^2 d\rho(\lm) \right) \left( \int_\R d\rho(\lm) \right) = -r''(0).
\end{align*}
Equality holds if and only if the function $\lm\mapsto \lm$ is a constant in $L^2(\rho)$, which is possible only if $\rho$ is a single atom (degenerate). Thus, by our hypothesis, $r'(0)^2 - r''(0) > 0$,
which concludes the proof of item 4.
\end{proof}

We shall also use the following integrability lemma.

\begin{lem}\label{lem: Fub}
Let $f:\R\to\C$ be a non-degenerate GSP, which is almost surely continuously differentiable, and let $T>0$. Then:
\begin{enumerate}[{\rm (I)}]
\item\label{item: Fub 1d}
$\ds \int_0^T \E\left[\Big| \frac{f'(t)}{f(t)}\Big|\right] \,dt < \infty.$
\item \label{item: Fub 2d}
$\ds \int_0^T \int_0^T \E  \left[\Big| \frac{f'(t) \ f'(s)} {f(t) \ f(s)}\Big|\right] \, dt\ ds < \infty.$
\end{enumerate}
Moreover, the integrand in the second item is divergent only at points $(t,s)$ where $r(t-s)=1$, and the two-dimensional Lebesgue measure of the set of such points is zero.
\end{lem}

This lemma first appeared in~\cite{Feld}*{Lemma 3.4}, and though it is stated there for functions that are a.s. analytic, it applies in our setting with no changes to the proof.

Our last lemma is an elementary but useful change of variables.

\begin{lem}\label{lem: int}
For $h\in L^1([-T,T])$ we have
\[
\int_0^T \int_0^T h(t-s) dt \, ds = \int_{-T}^T (T-|x|) h(x) dx.
\]
\end{lem}

\begin{proof}
\begin{align*}
\int_0^T \int_0^T h(t-s) dt \, ds &= \int_0^T \int_{-s}^{T-s} h(x) dx ds \\
&= \int_{-T}^T \int_{0\wedge -x }^{T\wedge (T-x)}  h(x) ds \, dx
= \int_{-T}^T (T-|x|) h(x) dx.
\end{align*}
\end{proof}

\subsection{The mean.} We now prove Theorem~\ref{thm: var}, part 1.
Applying the first item of Lemma~\ref{lem: K} and recalling \eqref{eq: def D} (and the normalisation $r(0)=1$) we have
\begin{align*}
\E [\D(T)] &=\frac 1{2i} \E \left[ \int_0^T \left(\frac{f'(t)}{f(t)} - \overline{\left(\frac{f'(t)}{f(t)}\right)}\right)dt \right] \\
&\overset{(*)}{=}\frac 1{2i} \int_0^T\E \left[ \frac{f'(t)}{f(t)} - \overline{\left(\frac{f'(t)}{f(t)}\right)}\right] dt \\
& = \frac{1}{2i} \int_0^T \left[ r'(0)  - \overline{r'(0)}\right] dt \\
&= T\cdot \Ima r'(0).
\end{align*}
The equality marked by (*) is justified by Fubini's theorem, which may be applied due to the first item of Lemma~\ref{lem: Fub}.

\subsection{The variance.}\label{sec: var form}
Here we prove part 2 of Theorem~\ref{thm: var}, pertaining to the variance.
By the definition of $\D(T)$ in \eqref{eq: def D}, we have
\begin{align}\label{eq: var long}
\V &= -\frac 1 4 \E \int_0^T \int_0^T \left( \frac{f'(t)}{f(t)} - \overline{\left(\frac{f'(t)}{f(t)}\right)} \right)
 \left( \frac{f'(s)}{fs)} - \overline{\left(\frac{f'(s)}{f(s)}\right)} \right) dt \, ds\notag\\
 &\qquad\qquad\qquad\qquad\qquad\qquad+\frac 1 4 \left( \E \int_0^T  \left( \frac{f'(t)}{f(t)} - \overline{\left(\frac{f'(t)}{f(t)}\right)} \right) dt \right)^2  \notag \\
& \overset{(**)}{=} -\frac 1 4 \int_0^T  \int_0^T \cov\left[ \frac{f'(t)}{f(t)} - \overline{\left(\frac{f'(t)}{f(t)}\right)},
\frac{f'(s)}{f(s)} - \overline{\left(\frac{f'(s)}{f(s)}\right)}\right]dt \ ds \notag \\
& = -\frac 1 4 \int_0^T  \int_0^T \bigg( \cov\left[\frac{f'(t)}{f(t)},\frac{f'(s)}{f(s)} \right]
-\cov\left[\frac{f'(t)}{f(t)}, \overline{\left( \frac{f'(s)}{f(s)}\right) } \right] \\
& \qquad\qquad\qquad -\cov\left[ \overline{\left( \frac{f'(t)}{f(t)}\right) },\frac{f'(s)}{f(s)} \right]
+\cov\left[\overline{ \left(\frac{f'(t)}{f(t)}\right) }, \overline{ \left(\frac{f'(s)}{f(s)}\right) } \right] \bigg) dt \ ds. \notag
\end{align}
The exchange of the order of the operations in the equality marked (**) is justified by Fubini's Theorem, which may be applied due to Lemma~\ref{lem: Fub}.
This lemma also allows us to ignore points $(t,s)$ where $|r(t-s)|=1$, and we shall do so for the remainder of the proof.

Next we apply Lemma~\ref{lem: K} in order to express each of the four terms in~\eqref{eq: var long} using the covariance function $r$ and its derivatives. We use~\eqref{eq: symmetries} to simplify our expressions. If $r(t-s) \neq 0$ then item (b) of Lemma~\ref{lem: K} implies that
\begin{equation}\label{eq: var no bar}
\cov\left[\frac{f'(t)}{f(t)},\frac{f'(s)}{f(s)} \right] =
\frac {|r(t-s)|^2 }{ 1- |r(t-s)|^2} \left( \frac{r'(t-s)}{r(t-s)} - r'(0) \right)
\left( \frac{r'(s-t)}{r(s-t)} - r'(0) \right),
\end{equation}
while item (c) implies that
\begin{align}\label{eq: var bar}
\cov\left[\frac{f'(t)}{f(t)}, \overline{ \frac{f'(s)}{f(s)} }  \right]& =
\frac {|r(t-s)|^2 }{ 1- |r(t-s)|^2} \left( \frac{r'(t-s)}{r(t-s)} - r'(0) \right)
\overline{ \left( \frac{r'(s-t)}{r(s-t)} - r'(0) \right) } \\
& + \log \frac 1{1-|r(t-s)|^2} \left( - \frac {r''(t-s)}{r(t-s)} - \frac {r'(t-s)\overline{r'(s-t)}}{r(t-s)^2} \right). \notag
\end{align}
At the points where $r(t-s)=0$, we have (by the same lemma)
\begin{equation}\label{eq: when zero}
\cov\left[\frac{f'(t)}{f(t)},\frac{f'(s)}{f(s)} \right] = - |r'(t-s)|^2,
\quad \cov\left[\frac{f'(t)}{f(t)},\ov{\frac{f'(s)}{f(s)} } \right] =0.
\end{equation}

We are now ready to plug~\eqref{eq: var no bar}, \eqref{eq: var bar} and~\eqref{eq: when zero} into~\eqref{eq: var long}, and use the symmetry relations~\eqref{eq: R symmetries} in order to simplify our expressions.
We get
\begin{equation}\label{eq: var K}
\V = \int_0^T \int_0^T K(t-s) \,dt \, ds
\end{equation}
where
\begin{equation*}
K(x)=
\begin{cases}
\frac 1 2 |r'(x)|^2, & r(x)=0 \\
-\frac 1 4 \left[\frac{|r(x)|^2}{1-|r(x)|^2} H(x)
-\frac 1 4 \log \frac 1{1-|r(x)|^2} \left\{ R'(x) +\ov{R'(x)}\right\} \right], & 0<|r(x)|<1.
\end{cases}
\end{equation*}
Here
\begin{align*}
H(x)& =  (R(x)-R(0))(R(-x)-R(0)) - (R(x)-R(0))\ov{(R(-x)-R(0))}\\
&\qquad - (R(-x)- R(0))\ov{(R(-x) - R(0))} +\ov{(R(x) -R(0))}\:\ov{(R(-x)- R(0))} \\
& = -2|R(x) - R(0) |^2 + (R(x) - R(0))^2 + \ov{(R(x) - R(0)}^2\\
& = -4 [ \Ima (R(x)-R(0) )]^2 <0.
\end{align*}
Further simplifications yield that, for $x$ such that $0<|r(x)|<1$,
\begin{align*}
K(x) &=  \frac{|r(x)|^2}{1-|r(x)|^2 } \left(\, \Ima^2 \{R(x)-R(0)\} \right) - \frac 1 4 \log \frac 1 {1-|r(x)|^2} \cdot 2 \Rea R'(x) \\
&= \frac{|r(x)|^2}{1-|r(x)|^2 } \left(R(x)+R(-x)-2R(0)\} \right)^2 -\frac 1 2 \log \frac 1 {1-|r(x)|^2} \cdot  \Rea \{R'(x)\}.
\end{align*}

Notice the formula we obtained for $K$ coincides with the definition in~\eqref{eq: K}. Once we show that $K\in L^1([-T,T])$ for any $T$, we may use Lemma~\ref{lem: int} to pass from the double integral in~\eqref{eq: var K} to a single integral. This yields
\[ \V = T \int_{-T}^T \left(1- \frac{|x|}{T}\right) K(x) dx,\]
which is precisely the identity \eqref{eq: var integral in statement} in Theorem~\ref{thm: var}.

It remains only to prove that $K$ is locally integrable.
Denote $D=\{t:\ |r(t)|=1\}$. This set is discrete by item 3 of Lemma~\ref{lem: |r|=1}.
It is straightforward to check that $K$ is continuous at any point not in $D$,
so we need only prove integrability of $K$ on a neighborhood of an arbitrary point in $D$.
Let us first consider $0\in D$. Since $r$ is continuous, there is an $\ep>0$  such that $|r(x)| >\frac 1 2$ for all $|x|< \ep$. Therefore for $|x|<\ep$ we have $1-|r(x)|^2 \simeq x^2$, $R(x)+R(-x) -2 R(0) \simeq x^2$, and $\Rea R'(x) \simeq \Rea R'(0)$, which yields
\begin{align}\label{eq: integ bnd}
|K(x)| &\le
 \frac{|r(x)|^2}{1-|r(x)|^2}\ \left| R(x)+R(-x) -2 R(0) \right|^2
+  \frac 1 2 \log \frac 1 {1-|r(x)|^2} \left| R'(x) \right|^2 \notag \\
 & \lesssim x^2 +\log \frac 1 x,
\end{align}
which is integrable in $(-\ep,\ep)$.
Using item 4 of Lemma~\ref{lem: |r|=1}, this argument may be repeated to show integrability near any other point $t_m\in D$.


\section{An alternative form for the variance and a linear lower bound: Theorem~\ref{thm: asym}}\label{sec: asym}
The main goal of this section is to prove Theorem~\ref{thm: asym} concerning a linear lower bound on the variance. However,
most of the section will be devoted to prove the following reformulation of the second part of Theorem~\ref{thm: var}, from which Theorem~\ref{thm: asym} will follow rather easily.

\begin{prop}\label{prop: IBP}
Define $\widetilde{K}:\R\to\R$ by
\begin{equation}\label{eq: K tilde}
\widetilde{K}(x)=\begin{cases}
\frac 1 2 |r'(x)|^2, & \text{if } |r(x)|=0\text{ or }1,\\
 \frac{|r(x)|^2}{1-|r(x)|^2}\ \Ima^2 \left\{R(x)-R(0) \right\} + \frac 1 4 \left( \log \frac 1 {1-|r(x)|^2} \right)^{\prime} \left( \log \{|r(x)|^2\} \right)^{\prime},
&  \text{if } 0 < |r(x)| < 1.
\end{cases}
\end{equation}
Then $\widetilde{K}$ is integrable on any compact subset of $\R$, and
\begin{equation}\label{eq: var tilde}
\frac {\var [\D(T) ]}{T} = \int_{-T}^T \left( 1-\frac{|x|}{T} \right) \widetilde{K}(x) dx + \frac 1 {2T}\int_{|r(T)|^2}^1 \log\frac1{1-y} \frac{dy}y.
\end{equation}

\end{prop}

A few remarks are in order before we proceed with the proofs.
\begin{rmk}\label{rmk: for IBP}
{\rm
Notice that all terms in this expression are non-negative. It is interesting to note that $\widetilde{K}$ can be defined if $r$ is only once differentiable, and suggests that~\eqref{eq: var tilde} may continue to hold in this case. (The random variable $\Delta(T)$ can be defined if $f$ is simply continuous.)
}
\end{rmk}

\begin{rmk}\label{rmk: cnts K tilde}
{\rm
While the expression \eqref{eq: K tilde} is a natural one from the perspective of our proof, it has the obvious drawback that it is not a continuous function. On the other hand we may modify \eqref{eq: K tilde} on a (at most) countable set of points (which of course does not affect \eqref{eq: var tilde}) to get
\begin{equation*}
\widetilde{K}^{*}(x)=\begin{cases}
 |r'(x)|^2, & \text{if } r(x)=0,\\
 \frac{|r(x)|^2}{1-|r(x)|^2}\ \Ima^2 \left\{R(x)-R(0) \right\} + \frac 1 4 \left( \log \frac 1 {1-|r(x)|^2} \right)^{\prime} \left( \log \{|r(x)|^2\} \right)^{\prime},
&  \text{if } 0 < |r(x)| < 1\\
2(r'(0)^2-r''(0)), & \text{if } |r(x)|=1,
\end{cases}
\end{equation*}
which is a continuous function. ($\widetilde{K}^{*}$ and $\widetilde{K}$ differ on the sets $\{x:r(x)=0,r'(x)\neq0\}$ and $\{x:|r(x)|=1\}$. The first is obviously at most countable, while Lemma~\ref{lem: |r|=1} implies that the second is. Lemma~\ref{lem: |r|=1} (and the contents of its proof) and Observation~\ref{obs: sym} imply that $\widetilde{K}^{*}$ is continuous at the points where $|r|=1$.)
}
\end{rmk}

\subsection{Proof of Proposition~\ref{prop: IBP}}
First recall that, by \eqref{eq: var integral in statement},
\begin{equation*}
  \frac{\V}{T} =  \int_{-T}^T \left( 1-\frac{|x|}{T} \right) K(x) dx,
\end{equation*}
where $K$ is given by~\eqref{eq: K}.
Denoting
$
F=\{x\in \R:\  |r(x)| = 0\text{ or }1\} \},
$
and using the fact that (when it's defined) $\Rea R(x)=\Rea \frac {r'}{r}(x)=\frac 1 2 \left(\log |r(x)|^2 \right)^{\prime}$, we may rewrite \eqref{eq: K} as
\begin{equation}\label{eq: K 2}
K(x) =
\begin{cases}
\frac 1 2 |r'(x)|^2, & x \in F\\
 \frac{|r(x)|^2}{1-|r(x)|^2}\ \Ima^2 \left\{\frac {r'(x)}{r(x)}- r'(0) \right\}
 + \frac 1 4 \log \frac 1 {1-|r(x)|^2} \left(- \log \{|r(x)|^2\} \right)^{\prime \prime}, &  x\not\in F .
\end{cases}
 \end{equation}

Comparing this with \eqref{eq: K tilde}, we see that the only difference between $K$ and $\tilde K$ is in case $x\not\in F$, and this difference lies only in the second term of the relevant expressions. Thus, in order to prove Proposition~\ref{prop: IBP}, it is enough to study
\begin{equation}\label{eq: S_T}
  S_T=\int_{[-T,T]\setminus F} \left(1-\frac{|x|}{T}\right) \log \frac 1 {1-|r(x)|^2} \left(- \log \{|r(x)|^2\} \right)^{\prime \prime} dx.
\end{equation}

Notice that $(-T,T)\setminus F$ is an open set, and therefore the union of disjoint intervals. We perform integration by parts on each interval to rewrite $S_T$. The boundary terms are given by evaluation of the function
\[
  \left(1 - \frac {|x|}{T}\right) \log \frac 1 {1-|r(x)|^2}  \left(- \log |r(x)|^2 \right)^{\prime}
\]
at points where either $|x|=T$, or $|r(x)|=0$ or $1$, and this function vanishes at all such points.
Therefore, the integration by parts yields
\begin{align}\label{eq: S_T 2}
S_T & =\int_{[-T,T]\setminus F} \left( \left(1-\frac{|x|}T\right) \log \frac 1{1-|r(x)|^2}  \right)^\prime \, \left(\log|r(x)|^2 \right)^{\prime} dx \notag \\
& = \int_{[-T,T]\setminus F}  \left(1-\frac{|x|}T\right) \left( \log \frac 1{1-|r(x)|^2}  \right)^\prime \left(\log|r(x)|^2 \right)^{\prime} dx \notag \\
& \quad -\frac {2}{T} \int_{[0,T]\setminus F}  \log \frac 1{1-|r(x)|^2}  \, \left(\log|r(x)|^2 \right)^{\prime} dx,
\end{align}
where in the last step we used the fact that $|r|^2$ is even (see~\eqref{eq: symmetries}).
For the last term, we write\footnote{It might be the case that ``$N=\infty$'', i.e., that we have a countable number of points in $[0,T)$ where $r$ vanishes. We leave it to the reader to check that this does not affect the proof.} $[0,T)\setminus F=(t_0,t_1)\cup(t_2,t_3)\cup\dots\cup(t_{2N},t_{2N+1})$ where $0=t_0 < t_1 \leq t_2 < t_3\leq \dots \leq t_{2N} < t_{2N+1} =T$ and $|r(t_{2n-1})| = |r(t_{2n})|$ for $1\leq n \leq N$. We then have
\begin{align}\label{eq: small term}
 \int_{[0,T]\setminus F}\log  &\frac 1{1-|r(x)|^2}  \, \left(\log|r(x)|^2 \right)^{\prime} dx
=   \sum_{n=0}^N\int_{t_{2n}}^{t_{2n+1}} \log\left( \frac 1{1-|r(x)|^2} \right)\frac {\frac d{dx} \{|r(x)|^2\} }{|r(x)|^2} \ dx    \notag \\
&=\sum_{n=0}^N\int_{|r(t_{2n})|^2}^{|r(t_{2n+1})|^2} \log\left( \frac 1{1-y} \right)\cdot \frac {dy}{y} \qquad
\\ &= \int_{|r(0)|^2}^{|r(T)|^2} \log\left( \frac 1{1-y} \right)\cdot \frac {dy}{y}
= - \int_{|r(T)|^2}^{1} \log\left( \frac 1{1-y} \right)\cdot \frac {dy}{y}.\notag
\end{align}

Combining~\eqref{eq: S_T}, \eqref{eq: S_T 2} and \eqref{eq: small term} we arrive at
\begin{align*}
S_T =    \int_{[-T,T]\setminus F} \left(1-\frac{|x|}{T}\right)\left( \log \frac 1 {1-|r(x)|^2} \right)^{\prime} \left( \log \{|r(x)|^2\} \right)^{\prime}  dx + \frac {2}{T}\int_{|r(T)|^2}^1 \log\frac1{1-y} \frac{dy}y.
\end{align*}
This, combined with~\eqref{eq: var integral in statement} and~\eqref{eq: K 2} yields that
\begin{equation*}
\frac{\var [\D(T) ] }{T}= \int_{-T}^T \left( 1-\frac{|x|}{T} \right) \widetilde{K}(x) dx + \frac 1 {2T}\int_{|r(T)|^2}^1 \log\frac1{1-y} \frac{dy}y,
\end{equation*}
where $\widetilde{K}$ is given by \eqref{eq: K tilde}. Since, by Remark~\ref{rmk: cnts K tilde}, we may modify $\tilde K$ on a set of measure $0$ to yield a continuous function, integrability on a compact is trivial. Proposition~\ref{prop: IBP} follows.

\subsection{Proof of Theorem~\ref{thm: asym}}
By the third item of Lemma~\ref{lem: |r|=1}, there exists $\delta_0>0$ such that $(0,\delta_0)\subset \{x:  0 < |r(x)| < 1\}$. Notice that in the expression for the variance given by Proposition~\ref{prop: IBP}, namely in \eqref{eq: K tilde} and \eqref{eq: var tilde}, all of the terms are non-negative (cf. Remark~\ref{rmk: for IBP}). Therefore,
\[
\frac{\var [\D(T)] } {T}\ge \frac 1 4 \int_{0}^{\delta_0}  \left( \log \frac 1 {1-|r(x)|^2} \right)^{\prime} \left( \log \{|r(x)|^2\} \right)^{\prime} \,dx ,
\]
for all $T>\delta_0$. Write $g(x)=|r(x)|^2$. By the fourth item of Lemma~\ref{lem: |r|=1}, we have $g(x)= 1-Cx^2 + o(x^2)$ as $|x|\to 0$, with some $C>0$. Thus
\begin{align*}
\lim_{x\to 0} & \left( \log \frac 1 {1-|r(x)|^2} \right)^{\prime} \left( \log \{|r(x)|^2\} \right)^{\prime} =\lim_{x\to 0 }\frac{g^{\prime}(x)^2 }{(1-g(x))g(x)} = \lim_{x\to 0} \frac{4C^2 x^2 + o(x^2)}{Cx^2 + o(x^2)} = 4C >0,
\end{align*}
and so $\frac {g'^2}{(1-g)g}> 2C$ on some interval $(-\delta_1,\delta_1)$. Taking $\delta=\min\{\delta_0,\delta_1\}$, we obtain
\begin{equation*}
  \frac{\var [\D(T)] } {T}\ge  \frac{C\delta}4
\end{equation*}
for all $T\geq\delta$. The theorem follows.


\section{Linear Variance and CLT: Theorem~\ref{thm: CLT}}\label{sec: lin+clt}
In this section we prove Theorem~\ref{thm: CLT}. We begin with some observations regarding our premises.

\begin{obs}\label{obs: L2}
$r\in L^2(\R)$ if and only if the spectral measure $\rho$ has density $p(\lm)\ge 0 $ (w.r.t. the Lebesgue measure) such that $p\in L^2(\R)$.
Similarly, $r^{(k)}\in L^2(\R)$ if and only if $\rho$ has density $p(\lm)\ge 0$ such that $\lm^k p(\lm)\in L^2(\R)$.
\end{obs}

This observation follows from basic properties of Fourier transform.

\begin{obs}\label{obs: r'}
If $r, r'' \in L^2(\R)$, then $r'\in L^2(\R)$.
\end{obs}

\begin{proof}
By Observation~\ref{obs: L2}, the spectral measure has density $p\ge 0$ which satisfies $p(\lm)\in L^2(\R)$ and $\lm^2 p(\lm)\in L^2(\R)$.
Therefore
\[
\int_\R \lm^2 p^2(\lm) \ d\lm \le \int_\R \min(1, \lm^4) p^2(\lm) \ d\lm < \infty,
\]
which yields
$\lm p(\lm)\in L^2(\R)$. This is equivalent to $r'\in L^2(\R)$ by Observation~\ref{obs: L2}.
\end{proof}

\subsection{Linear Variance}\label{sec: lin var}
In this subsection we show the first part of Theorem~\ref{thm: CLT}, that is, that if $r, r'\in L^2(\R)$ then the variance of $\D(T)$ is asymptotically linear (in the sense of~\eqref{eq: lin limit}).

From Proposition~\ref{prop: IBP}, we have
\begin{equation}\label{eq: v}
\frac {\V}{T} = \int_{-T}^T \left( 1-\frac{|x|}{T} \right) \tilde {K}(x) \ dx + O\left(\frac 1 T\right),
\end{equation}
where $\tilde K\ge 0$ is given by~\eqref{eq: K tilde}. Since the integrand in~\eqref{eq: v} converges pointwise to $\tilde K (x)$ (as $T\to\infty$) and is dominated by it, by the dominated convergence theorem it is enough to show that
\begin{equation}\label{eq: |K| integ}
\int_{\R} \big|\tilde K(x)\big| \ dx < \infty.
\end{equation}
in order to get a finite limit as  $T\to \infty$ in~\eqref{eq: v}. Thus we need only show~\eqref{eq: |K| integ} in order to obtain~\eqref{eq: lin limit}.

By our premise, $|r(x)|\to 0$ as $|x|\to \infty$, and so there is $a>0$ such that $|r(x)|<\frac 1 2$ for $|x|>a$. By Proposition~\ref{prop: IBP}, $\tilde K$ is integrable on any finite interval, and in particular on $[-a,a]$. For $|x|>a$ and $r(x)\neq0$ we use the definition of $\tilde K$ in~\eqref{eq: K tilde} to get
\begin{align}\label{eq: |K| bound}
|\tilde K(x)| & \lesssim
 \frac{|r(x)|^2}{1-|r(x)|^2}\ \left| R(x)-R(0)\right|^2 \notag
+  \frac { \left(\frac{d}{dx} \{ |r(x)|^2 \} \right)^2} {(1-|r(x)|^2)|r(x)|^2}
\notag \\
 & \lesssim |r(x)|^2 |R(x)-R(0)|^2 + \frac {\big|r'(x)\ov{ r(x)} \big|^2} {|r(x)|^2}\notag \\
 & \lesssim |r(x)|^2 + |r(x)r'(x)| + |r'(x)|^2.
\end{align}
For $x$ such that $r(x)=0$ we have $\tilde K(x)=\tfrac 1 2 |r'(x)|^2$, so \eqref{eq: |K| bound} holds there as well. Since $r$ and $r'$ are both in $L^2(\R)$, each term on the left-hand side of~\eqref{eq: |K| bound} is integrable on $\R$. This proves~\eqref{eq: |K| integ}, and we are done.

\subsection{CLT}
In this section we prove the central limit law in Theorem~\ref{thm: CLT}. The main steps are as follows:
\begin{enumerate}
\item Construct an $M$-dependent stationary Gaussian process $f_M:\R\to \C$, that approximates the original process $f$ (in a way to be clarified). For this we employ an approximation strategy of Cuzick~\cite{Cuz}, although the idea goes back to Malevich~\cite{Mal}.

\item Show that the increment of the argument of $f_M$, denoted $\D_M(T)$, obeys a CLT as $T\to\infty$ for each fixed $M$.

\item Show that $(\D_M(T)-\E \D_M(T))/\sqrt{\var(\D_M(T))}$ approaches $(\D(T)-\E \D(T))/\sqrt{\var(\D(T))}$ as $M\to\infty$ in $L^2(\Pro)$, uniformly in $T$.
\end{enumerate}
These steps will conclude the proof, by the following standard lemma.
\begin{lem}\label{lem: tool for CLT}
Suppose that, for $T>0$ and $M>0$, $X(T)$ and $X_M(T)$ are real-valued random variables with mean $0$ and variance $1$. Suppose further that the following holds:
\begin{itemize}
  \item For each fixed $M$,
  \[
  X_M(T) \overset{d}{\longrightarrow} \calN_\R(0,1), \text{ as } T\to\infty.
  \]
  \item We have
  \[
  \lim_{M\to\infty}\E \left[ (X(T)-X_M(T))^2 \right] =0,
  \]
  uniformly in $T$.
\end{itemize}
Then
\[
 X(T) \overset{d}{\longrightarrow} \calN_\R(0,1), \text{ as } T\to\infty.
  \]
\end{lem}

\subsubsection{Constructing an approximating process}
In this section we construct a finitely-dependent process $f_M$ which will be close, in some sense, to $f$. We remind the reader that we are working under the hypothesis that $r,r'$ and $r''$ are in $L^2(\R)$. By Observation~\ref{obs: L2}, the spectral measure has density $p(\lm)\ge 0$ with respect to the Lebesgue measure. For what follows, it will be convenient to write $p(\lm)=g^2(\lm)$. We may represent the processes as an integral against white noise,
\begin{equation}\label{eq: f B}
f(t) = \int_\R e^{-i\lm t} g(\lm) \ d\W(\lm),
\end{equation}
where $\W = \frac 1 {\sqrt 2} (\B_1 + i \B_2)$ with $\B_1$ and $\B_2$ being two independent one-dimensional Brownian motions. For details on this representation see, for instance, \cite{CL}*{Ch. 7.5}.

For $M>0$, define
\begin{equation*}
P_M(\lm) = \frac M {K_1} \sinc^4(M\lm),
\end{equation*}
where $\sinc(\lm) = \frac{\sin (\pi \lm)}{\pi \lm}$, and
\footnote{We remark that one may compute $K_1 = \tfrac23$, though this value will be unimportant for our purposes.}
$K_1=\int_{-\infty}^{\infty}\sinc^4 (\lm) \ d\lm$. Note that $P_M(\lm)\ge 0$ for all $M>0$ and $\lm\in\R$, and that
\[
\int_{-\infty}^\infty P_M(\lm)\  d\lm =1.
\]
For each $M> 0$, we define a new process by
\[
f_M(t) = \int_\R e^{-i\lm t}\, (g^2 * P_M)^{1 / 2}(\lm) \ d\W(\lm),
\]
where $d\W (\lm)$ is the \underline{same} white noise as in \eqref{eq: f B} (that is, the processes $f$ and $f_M$ are coupled).
Indeed, the process $f_M$ is well defined since $(g^2 * P_M)(\lm)\ge 0$ for all $\lm\in\R$, and
\[
\lm \mapsto e^{-i\lm t} (g^2 * P_M)^{1/2}(\lm) \in L^2(\R).
\]
Further, $f_M$ is a Gaussian stationary process with spectral measure
\begin{equation*}
d\rho_M(\lm) = (g^2 *P_M)(\lm) \ d\lm,
\end{equation*}
and covariance kernel $r_{M,M}(t)=\E [ f_M(t)\ov{f_M(0)} ]$ satisfying
\begin{equation}\label{eq: rMM}
r_{M,M}(t) = \wh{\rho_M}(t) = r(t) \ \wh{P}_M(t).
\end{equation}
We further define $r_{0,M}(t) = \E[f(t)  \ov{f_M(0)}]$, and notice that
\begin{equation*}
r_{0,M}(t) = \int e^{-i\lm t} g(\lm) (g^2 * P_M)^{\frac 1 2}(\lm) d\lm.
\end{equation*}

\subsubsection{Properties}
In this subsection we clarify in what sense $f_M$ approximates $f$. More importantly we prove the following key result, which concerns the convergence of the covariance kernels $r_{M,M}$ and $r_{M,0}$, and will be essential in proving the CLT in Theorem~\ref{thm: CLT}.

\begin{prop}\label{prop: r limits}
  As $M\to\infty$, we have
  \begin{align*}
    &r_{0,M}, \, r_{M,M} \rightarrow r, \\
    &r'_{0,M}, \, r'_{M,M} \rightarrow r', \\
    &r''_{0,M}, \, r''_{M,M} \rightarrow r'',
  \end{align*}
  in both $L^2(\R)$ and $L^\infty(\R)$.
\end{prop}
We recall the definition of $M$-dependence.
\begin{defin*}[$M$-dependence]
Let $T\subseteq \R$, and $M\ge 0$.
A stochastic process $(X(t))_{t\in T}$ is \emph{$M$-dependent}
if for any $s_1,s_2\in T$ such that $s_2-s_1>M$, the sigma-algebras generated by $(X(t))_{t\le s_1}$ and $(X(t))_{t\ge s_2}$ are independent.
\end{defin*}
\begin{prop}\label{prop: f_m props}
The process $f_M$ is almost surely continuously differentiable, and $4\pi M$-dependent.
\end{prop}
Further, $f_M$ approximates $f$ in the following sense, which we immediately deduce from the previous two propositions.
\begin{cor}\label{cor: f limits}
As $M\to\infty$, we have
\[
f_M(t)\rightarrow f(t) \quad \text{and} \quad f'_M(t)\rightarrow f'(t)
\]
in $L^2(\Pro)$, uniformly in $t\in \R$.
\end{cor}

We will now give a series of lemmata and observations which will lead to the proof of the previous two propositions.
\begin{lem}\label{lem: PM}
$\wh{P}_M(t)$ satisfies the following properties:

\begin{enumerate}
\item $\wh{P}_M(t)$ is twice continuously differentiable on $\R$.
\item $0\le \wh{P}_M(t)\le 1$ for all $t\in\R$.
\item $\wh{P}_M(t)=0$ for $|t|>4\pi M$.
\item For any $0<\ep<1$ we have $\wh{P}_M(t) = 1-\frac{K_2}{M^2} t^2  + O\left(\frac {t^{2+\ep}}{M^{2+\ep}}\right)$, as $t\to 0$, where
\footnote{Again, it is possible to compute $K_2=\tfrac{3}{8\pi^2}$.}
    $ K_2 = \frac 1 {2K_1} \int_\R \lm^2 \sinc^4(\lm) \ d\lm$ and the implicit constant depends only on $\ep$.
\end{enumerate}
\end{lem}

\begin{proof}
The first item follows from the fact that
\begin{equation}\label{eq: ep moment}
  \int_\R |\lm|^{2+\ep} P_M(\lm) \ d\lm <\infty
\end{equation}
for all $0< \ep < 1$. Notice that $\wh{\sinc}(t) = \ind_{[-\pi,\pi]}(t)$. So
\begin{align*}
  \wh{P}_M(t) =
  \wh{P}_1\left(\frac t M\right) =
  \frac 1 {K_1} (\wh{\sinc^4}) \left(\frac t M \right)
   = \frac 1 {K_1} \ind_{[-\pi,\pi]}^{*4} \left(\frac t M\right),
\end{align*}
where $*4$ denotes the fourth convolution power. Clearly, $\wh{P}_M(t)\ge 0$ for all $t\in \R$. Also, by basic properties of the Fourier transform we have $\wh{P}_M(t) \le \wh{P}_M(0)=\int_\R P_M =1$, which establishes the second item. Further, since $\ind_{[-\pi,\pi]}^{*4}$ is supported on $[-4\pi,4\pi]$, the third item follows. Finally notice that
\begin{align*}
  & \wh{P}_1(0) = \int_\R P_1 = 1,  \\
  & \wh{P}_1^{\,'}(0) = -i\int_\R \lm P_1(\lm) \ d\lm = 0, \\
  & \wh{P}_1^{\,''}(0) = -\int_\R \lm^2 P_1(\lm) \ d\lm = -2K_2,
\end{align*}
and by a standard Taylor expansion, taking \eqref{eq: ep moment} into account, the fourth item follows.

\end{proof}

\begin{lem}\label{lem: rMM}
The following hold:
\begin{enumerate}
\item  $r_{M,M}$ is a twice differentiable function on $\R$, supported on $[-4\pi M, 4\pi M]$.
\item $r_{M,M}(0)=r(0)=1$ and $r'_{M,M}(0)=r'(0)$.
\item $|r_{M,M}(t)|\le |r(t)|$ for all $t$.
\end{enumerate}
\end{lem}
Recalling that $r_{M,M}(t)=r(t)\wh{P}_M(t)$ (see~\eqref{eq: rMM}), Lemma~\ref{lem: rMM} follows immediately from Lemma~\ref{lem: PM} and our assumptions about $r$.

This previous lemma immediately implies that $f_M$ is a $4\pi M$-dependent process. The next lemma will complete the proof of Proposition~\ref{prop: f_m props}.
\begin{lem}
  Suppose that \eqref{eq: log moment} holds for $\alpha>0$. Then
  \begin{equation*}
    \int_{\R} \lm^2 \log^{1+\al'} (1+|\lm|) d\rho_{M}(\lm)<\infty
  \end{equation*}
  for any $\al'<\al$.
\end{lem}
\begin{proof}
  Combining Item 4 of Lemma~\ref{lem: PM} with \eqref{eq: r'' at 0} we see that for all $C>0$ we have
  \begin{equation*}
    |r_{M,M}''(t)-r_{M,M}''(0)|\leq\frac{C}{|\log|t||^{\al}}\quad\text{for }|t|\leq\delta(\al,C).
  \end{equation*}
  Then \cite{CL}*{Chapter 9, Lemma 2} completes the proof.
\end{proof}

For Proposition~\ref{prop: r limits} we shall need two further lemmas about the kernel $P_M$.

\begin{lem}\label{lem: sum ker}
For any $1\le p<\infty$ and $h\in L^p(\R)$, we have $P_M * h\conv{L^p} h$ as $M\to\infty$.
\end{lem}

\begin{proof}
Observe that $(P_M)_{M>0}$ is a \emph{summability kernel}; that is,
$P_M(\cdot)\ge 0$,  $\int_\R P_M = 1$, and for every fixed  $\ep>0$ the convergence $\lim_{M\to\infty}\int_{|x|>\ep}P_M =0$ holds.
A standard property of summability kernels (see~\cite{Kat}*{Ch. VI}) establishes our lemma.
\end{proof}

\begin{lem}\label{lem: deriv P_M hat to 0}
  $\wh{P}_M^{\,'},\wh{P}_M^{\,''}\to0$ in $L^2$ and $L^{\infty}$ as $M\to \infty$.
\end{lem}

\begin{proof}
  Notice that $\wh{P}_M(t)=\wh{P}_1\left(\tfrac tM \right)$ and so $\wh{P}_M^{\,'}(t)=\tfrac 1M \wh{P}_1^{\,'}\left(\tfrac tM \right)$. This implies that
  \begin{equation*}
    \norm{\wh{P}_M^{\,'}}_2^2 = \frac 1{M^2} \int_{\R} \big|\wh{P}_1^{\,'}(\tfrac tM)\big|^2\,dt = \frac 1{M} \int_{\R} \big|\wh{P}_1^{\,'}(s)\big|^2\,ds \to 0
  \end{equation*}
  and
  \begin{equation*}
    \norm{\wh{P}_M^{\,'}}_\infty = \frac 1{M} \norm{\wh{P}_1^{\,'}}_\infty\to 0.
  \end{equation*}
  Similarly
  \begin{equation*}
    \norm{\wh{P}_M^{\,''}}_2^2 = \frac 1{M^3} \int_{\R} \big|\wh{P}_1^{\,''}(s)\big|^2\,ds \to 0,
  \end{equation*}
  and
  \begin{equation*}
    \norm{\wh{P}_M^{\,''}}_\infty = \frac 1{M^2} \norm{\wh{P}_1^{\,''}}_\infty\to 0.
  \end{equation*}
\end{proof}

We will also need two simple observations.
\begin{obs}\label{obs: Fourier contin}
The following hold:
\begin{itemize}
\item If $h_n\overset{L^1}{\rightarrow} h$, then $\wh h_n \overset{L^\infty}{\rightarrow} \wh h$.
\item If $h_n \overset{L^2}{\rightarrow} h$, then $\wh h_n\overset{L^2}{\rightarrow} \wh h$.
\end{itemize}
\end{obs}

\begin{obs}\label{obs: L1 to L2}
  If $h,h_n\ge 0$ and $h_n^2\overset{L^1}{\rightarrow} h^2$, then
  $h_n\overset{L^2}{\rightarrow} h$.
\end{obs}

\begin{proof}
  Since for any $x,y\geq0$ we have $|x-y|^2\leq |x^2-y^2|$ it follows trivially that
  \begin{equation*}
    \int |h_n - h|^2 \le \int | h_n^2 -h^2| \to 0.
  \end{equation*}
\end{proof}

\begin{proof}[Proof of Proposition~\ref{prop: r limits}]
Recall our hypothesis that $r,r'$ and $r''\in L^2$. By Observation~\ref{obs: L2} we deduce that
\begin{equation*}
  \lm^j g^2(\lm) \in (L^1 \cap L^2)(\R)
\end{equation*}
Since $P_M$ is a summability kernel, we may apply Lemma~\ref{lem: sum ker} to see that
\begin{equation}\label{eq: conv}
   \left(\lm^j g^2(\lm) \right) * P_M \rightarrow \lm^j g^2(\lm) \text{  in } L^1 \text{ and } L^2\quad\text{for }j=0,1\text{ and }2.
\end{equation}
By Observation~\ref{obs: Fourier contin} this implies that
\begin{equation}\label{eq: r(j)PM}
  r^{(j)}\wh{P}_M \to r^{(j)} \text{ in } L^2 \text{ and } L^\infty.
\end{equation}

Recalling that
\begin{equation*}
  r_{M,M}=r\wh{P}_M
\end{equation*}
we see that the case $j=0$ in \eqref{eq: r(j)PM} is equivalent to $r_{M,M}\to r$. Further, since $|r(t)|\leq r(0)=1$, we have $r\wh{P}_M^{\,'}\to0$ in $L^2$ and $L^{\infty}$ by Lemma~\ref{lem: deriv P_M hat to 0}. This, combined with the case $j=1$ in \eqref{eq: r(j)PM} and
\begin{equation*}
  r_{M,M}'=r'\wh{P}_M + r\wh{P}_M^{\,'},
\end{equation*}
implies that $r_{M,M}'\to r'$ in $L^2$ and $L^{\infty}$. Arguing similarly we have $r_{M,M}''\to r''$ in $L^2$ and $L^{\infty}$.

It remains to prove the convergence of $r_{0,M}$ and its derivatives. First recall that $\wh{r}_{0,M}= g (g^2 * P_M)^{\frac 1 2}$. We have
\begin{equation*}
  \int_{\R}\lm^2 g(\lm) (g^2 * P_M)^{\frac 1 2}(\lm) d\lm \leq \left(\int_{\R}\lm^2 g(\lm)^2 d\lm\right)^{\tfrac12} \left(\int_{\R}\lm^2 (g^2 * P_M)(\lm) d\lm\right)^{\tfrac12}<+\infty
\end{equation*}
and so $r_{0,M}$ is indeed twice differentiable. By Observation~\ref{obs: L1 to L2}, the $L^1$ convergence in~\eqref{eq: conv} with $j=0$ implies that
\begin{equation*}
  \left( g^2 * P_M \right)^{\frac 1 2} \rightarrow  g \text{  in } L^2.
\end{equation*}
Thus for $j=0,1$ and $2$
\begin{equation*}
  \norm{r^{(j)}_{0,M}-r^{(j)}}_\infty \le \norm{\lm^j g\big( (g^2*{P_M})^{\frac 1 2}-g\big) }_1 \le \norm{\lm^j g}_2 \cdot \norm{\left( g^2 * P_M \right)^{\frac 1 2}-g}_2 \to 0, \quad M\to \infty,
\end{equation*}
which proves the desired $L^\infty$ convergence. As for $L^2$ convergence, we have
\begin{align*}
\frac 1{\sqrt{2\pi}}\norm{r^{(j)}_{0,M}-r^{(j)}}_2 &= \norm{\lm^j g\big( (g^2*{P_M})^{\frac 1 2}-g\big)}_2 \\
&= \norm{ \lm^{2j} g^2 \big( (g^2* P_M)^{\frac 1 2}-g \big)^2}_1^{\frac 1 2} \\
& \le \norm{\lm^{j}g^2}_2^{\frac 1 2} \norm{ \big( \lm^{j/2} ( g^2* P_M)^{1/2}- \lm^{j/2} g \big)^2}_2^{\frac 1 2}\\
& \le \norm{\lm^{j} g^2}_2^{\frac 1 2} \norm{ \lm^j(g^2*P_M) -\lm^j g^2}_2^{\frac 1 2}.
\end{align*}
The last inequality follows from the observation that $\norm{(\psi-\phi)^2}_2\le \norm{\psi^2-\phi^2}_2$ for $\psi, \phi\ge 0$. Now
\begin{equation*}
  \norm{ \lm^j(g^2*P_M) -\lm^j g^2}_2 = \frac 1{\sqrt{2\pi}}\norm{r_{M,M}^{(j)} - r^{(j)}}_2\to 0,\quad M\to\infty,
\end{equation*}
which completes the proof.

\end{proof}

\subsubsection{CLT for the approximating process}

In this subsection we prove that $\D_M(T)$ satisfies a CLT as $T\to\infty$.
\begin{prop}\label{prop: D_M clt}
For each fixed $M\ge 0$,
\[
\frac{\D_M(T)-\E[\D_M(T)]}{\sqrt{\var [\D_M(T)] }} \rightarrow \calN_\R(0,1)
\]
in distribution as $T\to\infty$.
\end{prop}

Our main tool is the following theorem of Diananda~\cite{Dian}*{Theorem 4}, which
guarantees a CLT for sums of $M$-dependent sequences.
\begin{thm}[Diananda]\label{thm: Dian}
  Let $(X_n)_{n\in\N}$ be an $M$-dependent sequence of identically distributed real-valued random variables, with mean $0$ and finite variance.
  Define $S_N = \sum_{n=1}^N X_n $, 
  and suppose that the variance of $S_N$ is at least linear, that is,
  \[
  \liminf_{N\to\infty} \frac {\var (S_N) }{N} > 0.
  \]
  Then
  \[
  \frac {S_N} {\sqrt{\var(S_N)}} \overset{N\to\infty}{\longrightarrow} \calN_\R(0,1)
  \]
  in distribution.
\end{thm}

Applying it, and accounting for differences between discrete and continuous time, we now prove our proposition.
\begin{proof}[Proof of Proposition~\ref{prop: D_M clt}]
Fix $M> 0$ throughout the proof.
Since $f_M$ is a $(4\pi M)$-dependent process, so too is the sequence
\[
\D_{M,n} =\int_{n-1}^n \Ima\left\{\frac{f'_M(t)}{f_M(t)}\right\} \ dt, \quad n\in\N.
\]

We may apply Diananda's Theorem to $\ov{\D}_{M,n}= \D_{M,n}-\E\D_{M,n}$, as this is a sequence of identically distributed random variables (due to stationarity), and
$ \var\left(\sum_{n=1}^N \D_{M,n} \right) = \var \ \D_M(N) \sim C_M N$ (by Section~\ref{sec: lin var}). Let $T\ge 0$, $N=\lfloor T \rfloor$ and define $E_{M}(T)= \D_M(T) - \sum_{n=1}^N \D_{M,n}$. We now show that the error $E_M(T)$ is small in an appropriate sense, which allows us to conclude asymptotic normality for $\D_M(T)$.

Denote $\ov{E}_M(T) = E_M(T)-\E E_M(T)$ and notice that
\[
E_{M}(T) =\int_{N}^T \Ima\left\{\frac{f'_M(t)}{f_M(t)}\right\} \ dt.
\]
Asymptotic normality of $\D_M(T)$ will follow if we show that
\[
\frac{\ov{E}_M(T)}{\var \ \D_M(T)} \to 0 \text{ in probability, as } T\to\infty,
\]
and
\[
\frac{\var\left(\sum_{n=1}^N \D_{M,n} \right)} {\var \ \D_M(T) } \rightarrow 1, \text{ as } T\to\infty.
\]

Since $ \var\ \D_M(T) \sim C_M T$ as $T\to \infty$, we have
$\frac{\var\left(\sum_{n=1}^N \D_{M,n} \right)} {\var \ \D_M(T) } =
\frac {\var \ \D_M(N)} {\var \ \D_M(T)} \sim \frac N T \sim 1$, as $T\to\infty$.

Moreover, by stationarity
$ \var \ E_M(T)  \le \sup_{0\le t\le 1}\var \D_M(t) = c_M <\infty $,
and so for each fixed $\ep>0$,
\[
\Pro \left( \left| \frac{\ov{E}_M(T) }{\sqrt{\var \ \D_M(T)} } \right| > \ep \right) \le \frac 1 {\ep^2} \frac {\var\ E_M(T)}{\var \ \D_M(T)} \rightarrow 0,
\]
as $T\to\infty$.
The proof is complete.
\end{proof}

\subsubsection{Quantifying the approximation}
In this section we show that, when appropriately normalized, ${\D}_M(T)$ approaches $\D(T)$ in $L^2(\Pro)$ as $M\to\infty$, uniformly in $T$. This is stated precisely in the following proposition.
For brevity, we write $\ov{\D}(T)= \D(T)-\E\D(T)$ and $\ov{\D}_M(T)=\D_M-\E \D_M(T). $

\begin{prop}\label{prop: goal}
There exists $T_0>0$ satisfying the following:
Given $\ep>0$, there exists $M_0=M_0(\ep)>0$ such that for all $T\ge T_0$ and $M\ge M_0$ we have
\begin{equation*}
\E \left[  \ \left( \frac{\ov \D_M(T)}{\sqrt{\var \D_M(T)}}
- \frac{\ov \D(T)}{\sqrt{\var \D(T)}}\right)^2 \   \right] < \ep.
\end{equation*}
\end{prop}

In fact, given our previous variance computations, it is enough to prove the following.

\begin{prop}\label{prop: goal T}
There exists $T_0>0$ satisfying the following:
Given $\ep>0$, there exists $M_0=M_0(\ep)>0$ such that for all $T\ge T_0$ and $M\ge M_0$ we have
\begin{equation*}
\frac 1 T \E \left[  \ \left( \ov \D_M(T)-\ov \D(T)\right)^2 \   \right] < \ep.
\end{equation*}
\end{prop}

%

\begin{proof}[Proof of Proposition~\ref{prop: goal T}]
  Repeating the arguments of Section~\ref{sec: var form} and using Lemma~\ref{lem: rMM}, we get
  \[
  \E \left[  \ \left( \ov \D_M(T)-\ov \D(T)\right)^2 \   \right]
  = T \int_{-T}^T \left(1-\frac{|x|}{T} \right) \Big( K(x) - 2 K_{0,M}(x) + K_{M,M}(x) \Big) \ dx,
  \]
  where $K(x)$ is as before (recall~\eqref{eq: K}), and
  \[
  K_*(x) =
  \begin{cases}
    \frac 1 2 |r'_*(x)|^2, & \text{if } r_*(x)=0\\
    \frac{|r_*(x)|^2}{1-|r_*(x)|^2}\ \Ima^2 \left\{\frac {r'_*(x)}{r_*(x)}- r'(0) \right\}
    -\frac 1 4 \log \frac 1 {1-|r_*(x)|^2}\ \frac {d^2}{dx^2} \left( \log |r_*(x)|^2 \right) , &  \text{otherwise.}
  \end{cases}
  \]
  Here $*$ may be either the index $(0,M)$ or $(M,M)$.
  We leave $K$, $K_{0,M}$ and $K_{M,M}$ undefined at $x=0$.
  We therefore have
  \[
  \frac 1 T \E \left[ \left( \ov \D_M(T) -\ov \D(T) \right)^2 \right]  \le  \int_\R  \left| K(x) - 2 K_{0,M}(x) + K_{M,M}(x) \right| \ dx.
  \]
  Thus it is enough to show that
\begin{equation}\label{eq: KM to K}
\textbf{Goal: }\qquad K_{0,M} \conv{L1} K \quad \text{ and }\quad K_{M,M} \conv{L1} K.
\end{equation}

We shall give a unified proof for both limits. We use the notation $(r_M, K_M)$ to represent either the pair $(r_{0,M},K_{0,M})$ or the pair $(r_{M,M},K_{M,M})$. We will repeatedly use that, by Proposition~\ref{prop: r limits}, $r_M, r'_M$ and $r''_M$ converge to $r,r'$ and $r''$ respectively, in $L^\infty(\R)$.
(Notice that by Lemma~\ref{lem: |r|=1}, $K(x)$ is a finite number whenever $x\ne 0$.)
In particular, this implies that
\begin{equation}\label{eq: point conv}
  K_M(x) \longrightarrow K(x),\qquad\text{pointwise},
\end{equation}
for $x\ne 0$.

We will prove $L^1$ convergence separately on two different subsets of $\R$, which are chosen according to the size of $r$. Fix some $\delta>0$ (to be specified later), and write
\begin{align*}
  A & = \{ x: \: |r(x)|\ge \delta \}\setminus\{0\} \\
  B & = \{ x: \: |r(x)| < \delta \}
\end{align*}

We begin with the convergence on $A$. Notice that $A$ is bounded (since $\lim_{x\to\pm\infty}|r(x)|=0$).
Denote $R_M = \frac {r'_M}{r_M}$. Similar to~\eqref{eq: integ bnd}, we have
\begin{align*}
|K_M(x)| &\le
 \frac{|r_M(x)|^2}{1-|r_M(x)|^2}\ \left| R_M(x)+R_M(-x) -2 R(0) \right|^2
+  \frac 1 2 \log \frac 1 {1-|r_M(x)|^2} \left| R_M'(x) \right|^2 \\
 & \le  L\cdot (x^2 +\log \frac 1 x ),
\end{align*}
where $L$ is a constant, uniform in $M$. Thus $K_M$ has an integrable majorant on $A$ and so \eqref{eq: point conv} and the dominated convergence theorem yield, for any fixed $\delta>0$,
\begin{equation}\label{eq: conv A}
  \int_A \left|K_M-K\right| \to 0.
\end{equation}

Next we turn to convergence on $B$.
Note that, since $r_M\to r$ in $L^\infty$, we may assume that $|r_M(x)| \le 2\delta$ for all $x\in B$.
Using the Taylor expansions
\[
\frac{1}{1-|r_M|^2} = 1+O(|r_M|^2), \quad
\log \frac{1}{1-|r_M|^2} = |r_M|^2 +O(|r_M|^4),
\]
we have for $x\in B$ and $r_M(x)\neq0$
\begin{align}\label{eq: big expand}
  K_M(x) = & -\frac 1 4 \frac{|r_M(x)|^2 }{1-|r_M(x)|^2}  \Bigg[ \left(\frac{r'_M(x)}{r_M(x)} \right)^2 + \ov{\left(\frac{r'_M(x)}{r_M(x)} \right)}^2 -  2\left| \frac{r'_M(x)}{r_M(x)} \right|^2
-4 r'(0) \frac{r'_M(x)}{r_M(x)}\notag \\
&\qquad\qquad\qquad\qquad\qquad + 4 r'(0) \ov{ \left(\frac{r'_M(x)}{r_M(x)} \right) } + 4 r'(0)^2 \Bigg] \notag \\
  &+ \frac 1 4 \log \frac 1 {1-|r_M(x)|^2}
  \left[ \left(\frac{r'_M(x)}{r_M(x)} \right)^2 + \ov{\left(\frac{r'_M(x)}{r_M(x)}  \right)}^2
  - \frac{r''_M(x)}{r_M(x)} -\ov{\left( \frac{r''_M(x)}{r_M(x)} \right)}
\right]  \notag \\
= &\frac 1 2 |r_M'(x)|^2 -  r'(0)^2  |r_M(x)|^2 + r'(0) \ov{ r_M(x)} r'_M(x)-  r'(0) r_M(x) \ov{r'_M(x)} \\
& -\frac 1 4 r''_M(x)\ov{r_M(x)}-\frac 1 4 \ov{r''_M(x)}r_M(x) \notag \\
& + O( |r_M(x) r'_M(x)|^2 + |r_M(x)|^4 + |r'_M(x)r_M(x)^3| + |r''_M(x)r_M(x)^3| ) \notag.
\end{align}

Notice that when $r_M(x)=0$ we have simply $K_M(x) = \frac 1 2 |r'_M(x)|^2$, which also coincides with~\eqref{eq: big expand} in this case.
The same arguments may be applied to get an expression for $K(x)$, given by replacing $r_M$ by $r$ everywhere in the right-hand side of~\eqref{eq: big expand}. Therefore, we have

\begin{align*}
\int_B |K_M-K| & \le  \frac 1 2 \int_B \big| |r'_M|^2 -|r'|^2 \big|
+ |r'(0)|^2 \int_B \big| |r_M|^2-|r|^2\big| \\
& +2|r'(0)|\int_B\big|\ov{r_M} r'_M-\ov{r}r'\big|
 +\frac 1 2 \int_B |\ov{r_M} r_M''-\ov{r}r''| \\
& + \delta^2 O\big(\norm{r'}_2^2  + \norm{r}_2^2  + \norm{r'}_2\norm{r}_2 + \norm{r''}_2\norm{r}_2\big)
\end{align*}

Let $\ep>0$ be given.
By the $L^2$ convergences in Proposition~\ref{prop: r limits}, and by observing that if $\psi_n\conv{L^2}\psi$ and $\phi_n\conv{L^2}\phi$ then
$\psi_n \phi_n \conv{L^1} \psi \phi$, we get that for large enough $M$
\[
\int_B |K_M-K| \le \ep + c_0 \delta^2 ,
\]
where $c_0$ is a constant (depending only on $r$). Thus, by choosing $\delta=\delta(\ep)$ appropriately, we get
\[
 \int_B \left| K_M(x)-K(x) \right|  \ dx <2\ep,
\]
for large enough $M$. Together with~\eqref{eq: conv A}, this establishes the goal~\eqref{eq: KM to K}.

\end{proof}

\subsection{Conclusion: Proof of the CLT in Theorem~\ref{thm: CLT}}
At last, we conclude the proof of the central limit theorem appearing in~\eqref{eq: clt}.
We apply Lemma~\ref{lem: tool for CLT} with $X(T) = \D (T)$ and $X_M(T)=\D_M(T)$.
The first condition (a CLT for $\D_M$) is guaranteed by Proposition~\ref{prop: D_M clt}.
The second condition (a uniform $L^2$ approximation) is guaranteed by Proposition~\ref{prop: goal}.
Thus Lemma~\ref{lem: tool for CLT} implies that $\D(T)$ satisfies a CLT in the sense of~\eqref{eq: clt}, and we are done.

\section{Sub-quadratic variance: Theorem~\ref{thm: sub-quad} }\label{sec: sub-quad}
Lastly, we include the proof of Theorem~\ref{thm: sub-quad}.

\begin{proof}[Proof of Theorem~\ref{thm: sub-quad}]

Since $\rho$ has no atoms, $f$ is an ergodic process (this is the classical Fomin-Grenander-Maruyama theorem, see~\cite{F0}*{Theorem 4} and the references therein). By standard arguments, this also implies that $\D(T)$ is ergodic. Recall that the first and second moment of $\D(T)$ are finite (this is precisely Lemma~\ref{lem: Fub}).
Thus, by Von-Neumann's ergodic theorem, we have
\begin{equation}\label{eq: erg}
\lim_{T\to\infty}\frac {\D(T)} T  = \E \Delta(1),
\end{equation}
where the convergence is both in $L^1$ and $L^2$ (see~\cite{Walters}*{Cor. 1.14.1}).
We conclude that
$$ \lim_{T\to\infty}\frac {\V} {T^2}  =0.$$

\end{proof}

\vspace{5pt}
{\bf{Acknowledgements:}} We are grateful to Mikhail Sodin for suggesting the project and for useful discussions. We thank Igor Wigman for many insightful comments, and Baruch Horovitz for a detailed conversation about the motivations coming from phyiscs.

\end{document}